\documentclass[12pt]{article}

\usepackage{graphicx, amssymb, latexsym, amsfonts, amsmath, lscape, amscd,
amsthm, color, epsfig, mathrsfs, tikz, enumerate, xcolor}
\usepackage[normalem]{ulem} %*** to strike out text

\newtheorem{theorem}{Theorem}[section]
\newtheorem{conjecture}[theorem]{Conjecture}

\newtheorem{lemma}[theorem]{Lemma}

\newtheorem{observation}[theorem]{Observation}

\newcommand\Z{\mathbb{Z}}
\newcommand\supp{\operatorname{supp}}
\newcommand\range{\operatorname{range}}
\begin{document}

%\begin{frontmatter}

\title{{\bf Bouchet's conjecture for cyclically 5-edge-connected, cubic signed graphs}}
\author{
{Kathryn Nurse} \\
\mbox{}\\
{\small \'Ecole Normale Sup\'erieure, Paris, France}\\
}

\date{\today}

\maketitle

\begin{abstract}
A 1983 conjecture of Bouchet states that every flow-admissible signed graph has a nowhere-zero six-flow. We prove this conjecture for cyclically five-edge-connected, cubic signed graphs.
\end{abstract}

\section{Introduction}

Nowhere-zero flows on signed graphs generalize nowhere-zero flows on graphs, and are motivated by embeddings of graphs in non-orientable surfaces \cite{Bouchet}. In 1983, Bouchet conjectured the following, which is a signed graph analogue of Tutte's famous 5-flow conjecture \cite{Tutte_1954}. 

\begin{conjecture}[Bouchet \cite{Bouchet}]\label{bcon}
    Every flow-admissible signed graph has a nowhere-zero 6-flow.
\end{conjecture}

Bouchet showed that his conjecture is true with 6 replaced by 216, and demonstrated that if true it is best possible. The current best result toward Conjecture \ref{bcon} is by DeVos, Li, Lu, Luo, Zhang, and Zhang who showed in 2021 that 6 can be replaced with 11 \cite{DLLZZ}. 

Many results toward this conjecture have been obtained under some connectivity assumptions, and we list some chronologically. In 1987, Khelladi showed that the 6 can be replaced with 18 under the assumption that $G$ is 4-connected \cite{Khelladi}. In 2005, Xu and Zhang proved that  Bouchet's conjecture is true for 6-edge-connected signed graphs \cite{XU2005335}. In 2011, Raspaud and Zhu demonstrated that a flow-admissible, 4-edge-connected signed graph has a nowhere-zero 4-flow \cite{RASPAUD2011464}. In 2014, Wu, Ye, Zang and Zhang showed that a flow-admissible, 8-edge-connected signed graph has a nowhere-zero 3-flow \cite{3flow8ec}. It is folklore that to prove Bouchet's Conjecture, it suffices to consider cubic signed graphs (since we could find no reference, we include a proof in Section \ref{S2} for completeness). Recently, DeVos, \v{S}\'{a}mal, and this author showed that a flow-admissible, 3-edge-connected signed graph has a nowhere-zero 8-flow \cite{devos2023nowherezero8flowscyclically5edgeconnected, myThesis}. In this paper we give another connectivity-based result that allows vertices of degree three -- we prove that Bouchet's Conjecture holds for cyclically 5-edge-connected cubic signed graphs. 

A graph is \emph{$k$-regular} when every vertex has degree $k$. A 3-regular graph is called \emph{cubic}, and a connected, 2-regular graph is called a \emph{cycle}. A graph is \emph{cyclically $k$-edge-connected} when every edge-cut separating two cycles contains at least $k$ edges.

\begin{theorem}\label{mainCyc5ecBou}
    Every flow-admissible, cyclically 5-edge-connected, cubic signed graph has a nowhere-zero 6-flow.
\end{theorem}

Our process can roughly be described as follows: A flow-admissible, cyclically 5-edge-connected cubic signed graph $G$ has a special nowhere-zero $\Z_6$-flow by a theorem of DeVos et al. \cite{DLLZZ}. Because $G$ is cubic, without loss we may assume every vertex has at most one edge carrying flow into it, and boundary in $\{0,6\}$. Preserving these properties, we contract all positive edges. A perfect matching in the contracted graph converts the $\Z_6$-flow to a 6-flow, which yields a 6-flow in $G$. The assumption that the graph is cyclically 5-edge-connected is only used to find the perfect matching -- we believe, by extending the method described here, that assumption could be reduced.

We cover basics in Section \ref{S2}, including defining signed graphs, and the folkloric result reducing Bouchet's conjecture to the cubic case. Definitions are introduced throughout, as they are needed. In Section \ref{S3}, we provide a lemma to convert a $\Z_6$-flow to a 6-flow under certain circumstances, and prove our main result.

\section{Basics}\label{S2}

\subsubsection*{Signed graphs}

All graphs considered are finite, and we allow parallel edges and loops. We use \cite{BondyJ.A.JohnAdrian2008Gt/J} for terms not defined here. Let $G=(V,E)$ be a graph. We denote the collection of edges incident to a vertex $v \in V$ by $\delta(v)$, with the convention that a loop incident to $v$ appears twice in $\delta(v)$. The \emph{degree} $d(v)$ of $v$ is then $|\delta(v)|$.  
A \emph{signed graph} is a graph $G$ with a \emph{signature} $\Sigma \subseteq E(G)$. We call an edge in $\Sigma$ \emph{negative}, and an edge in $E(G) \setminus \Sigma$ \emph{positive}. To orient a signed graph, intuitively we assign a direction exactly twice to every edge, so that every positive edge is oriented both toward an end and away from an end, while every negative edge is oriented either toward all, or away from all of its ends. Thus, an \emph{oriented} signed graph, also called a \emph{bidirected graph}, is a signed graph $G$ with an \emph{orientation}: a bipartition\footnote{with possibly empty parts} of $\delta(v)$ into $\delta^+(v)$ and $\delta^-(v)$ for every $v \in V(G)$ so that $\Delta_{v \in V(G)}\delta^+(v)$ is the set of positive edges of $G$\footnote{$\Delta$ denotes symmetric difference}\footnote{A positive loop-edge incident to $v$ occurs once in $\delta^+(v)$ and once in $\delta^-(v)$, while a negative loop-edge incident to $v$ occurs twice in one of $\delta^+(v)$ or $\delta^-(v)$.}. We say an edge in $\delta^+(v)$ is directed \emph{away} from $v$, while an edge in $\delta^-(v)$ is directed \emph{toward}~$v$.
\medskip

Given an oriented signed graph $G$, an additively-written abelian group $A$, and a function $f: E(G) \to A$, we define the \emph{boundary} $\partial f: V(G) \to A$ as $\partial f(v) = \sum_{e \in \delta^+(v)} f(e) - \sum_{e \in \delta^-(v)} f(e)$. One may think of $\partial f$ as being the deficit of $f$ at each vertex. We say that $f$ is an $A$-\emph{flow} or just a \emph{flow} whenever $\partial f = 0$, we say that $f$ is \emph{nowhere-zero} whenever $0 \not\in f(E)$, and we say that $f$ is a \emph{$k$-flow} if it is an integer-flow where $f(E) \subseteq \{0, \pm1, \pm2, \dots, \pm(k-1)\}$.

\medskip

In the next two paragraphs, we define two well-used operations on an oriented signed graph $G$: reversing an edge, which involves the orientation but not the signature of $G$; and switching at a vertex, which involves both the signature and (if present) the orientation of $G$. 

Let $e \in E(G)$ be an edge, and $v$ be an end of $e$. If $e \in \delta^+(v)$ (respectively $e \in \delta^-(v)$) then to \emph{reverse $e$ at $v$} is to remove $e$ from $\delta^+(v)$ (resp. $\delta^-(v)$) and add $e$ to $\delta^-(v)$ (resp. $\delta^+(v)$). To \emph{reverse the edge $e$} (when no end is specified) means to reverse $e$ at each of its ends\footnote{If $e$ is a loop-edge incident to $v$, then by ``reversing $e$ at $v$'' and ``reversing $e$'' we mean the same operation: If $e$ is a positive loop, we do nothing. If $e$ is a negative loop with both ends toward (away from) $v$, then we modify the orientation so that $e$ is a negative loop with both ends away from (toward)~$v$.}. Notice that if $f: E(G) \to \Z$ is a nowhere-zero $k$-flow in an oriented signed graph $G$, then the oriented signed graph $G'$ formed from $G$ by reversing some $e\in E(G)$ also has a nowhere-zero $k$-flow $f'$, where $f'(d) = f(d)$ if $d \neq e$, and $f'(e) =-f(e)$.
Thus, to demonstrate the existence of a nowhere-zero $k$-flow in an oriented signed graph $G$, it suffices to demonstrate a nowhere-zero $k$-flow in any orientation of $G$. For simplicity, we do not always specify an orientation of a signed graph $G$ when speaking of a flow $f$ on its edges, the understanding, however, is that the reader may choose any arbitrary orientation\footnote{This generalizes to any group-valued function on the edges of a signed graph with prescribed boundary at each vertex)}.

Consider again the oriented signed graph $G$, and let $v \in V(G)$. To \emph{switch} at $v$ is to reverse at $v$ every edge incident to $v$, and simultaneously adjust the signature of $G$ so that every negative (positive) non-loop edge in $\delta(v)$ becomes positive (negative). 
Notice that if $G$ has a nowhere-zero $k$-flow $f$, then after switching at $v$, the function $f$ remains a nowhere-zero $k$-flow in the modified graph. In fact, the switching operation is fundamental to signed graphs, and we say two signed graphs $G$ and $G'$ are \emph{equivalent} if $G'$ can be obtained from $G$ by a sequence of switching operations. If $\Sigma$ is the signature of $G$, and $\Sigma'$ is the signature of $G'$, then we also say $\Sigma$ and $\Sigma'$ are equivalent signatures of $G$. To demonstrate the existence of a nowhere-zero $k$-flow in a signed graph $G$, it suffices to demonstrate a nowhere-zero $k$-flow in any signed graph $G'$ that is equivalent to $G$\footnote{This generalizes to any group-valued function on the edges of a signed graph with prescribed boundary \emph{up to inverse} at each vertex.}.

A signed graph $G$ is \emph{$k$-unbalanced} whenever every signed graph $G'$ equivalent to $G$ has at least $k$ negative edges. A $1$-unbalanced signed graph is simply \emph{unbalanced}. If a signed graph  is not unbalanced, it is \emph{balanced}.
We say a (signed) graph $G$ is \emph{flow-admissible} whenever there exists a function $f: E(G) \to \Z$ that is a nowhere-zero integer-flow. Bouchet \cite{Bouchet} proved that this is equivalent to the statement that for every edge $e \in E(G)$, $G-e$ has the same number of balanced components as $G$. 

\subsubsection*{Reducing Bouchet's conjecture to cubic signed graphs}

It is known that to prove Bouchet's conjecture it suffices to consider cubic signed graphs, however, since we could not find a proof in the literature we provide one for completeness.

Let $G$ be a signed graph with a vertex $v \in V(G)$ of degree two, and let $e,e' \in \delta(v)$ be distinct edges, whose other ends are $u,u'$. Then to \emph{suppress} $v$ is to delete $v$ from $V(G)$ and add a new edge $f$ to $E(G)$ with ends $u,u'$ so that $f$ is negative if and only if exactly one of $e,e'$ is negative.

\begin{observation}\label{deg2}
    Let $G$ be a signed graph with a vertex $v \in V(G)$ of degree two that is not incident to a loop. Then the graph formed from $G$ by suppressing $v$ has a nowhere-zero $k$-flow if and only if $G$ has a nowhere-zero $k$-flow.
\end{observation}

Let $G$ be a signed graph, $v \in V(G)$ have $d(v)\ge 4$, and $e_1,e_2 \in \delta(v)$. Let $v_1$ and $v_2$ be the other ends of $e_1$ and $e_2$ respectively (we allow loops). To \emph{uncontract at $v$ with $e_1,e_2$} is to add a new vertex $v'$ to $V(G)$, change the ends of $e_1$ and $e_2$ to be $v',v_1$ and $v',v_2$ respectively, and add a new positive edge with ends $v,v'$.

\begin{figure}[h]
    \centering
    \includegraphics[width=0.5\linewidth]{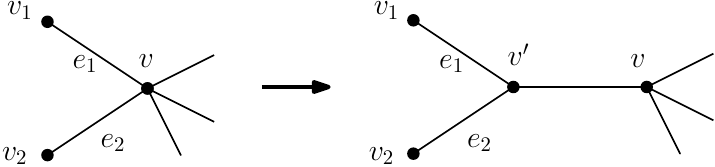}
    %\caption{Caption}
    \label{fig:placeholder}
\end{figure}

\begin{lemma}\label{uncont}
    Let $G$ be a flow-admissible signed graph, $v \in V(G)$ be a vertex of degree at least 4, and $e,e' \in \delta(v)$. Let $G'$ be the signed graph obtained from $G$ by uncontracting at $v$ with $e,e'$, and let $f$ be the uncontracted edge. Then one of $G'$ or $G' -f$ is flow-admissible.
\end{lemma}

\begin{proof}
    Because $G$ is flow-admissible, there exists a nowhere-zero integer flow $\phi: E(G) \to \Z$. If $\phi$ is also a nowhere-zero integer flow in $G'-f$, then we are done. If not, then $ 0 \neq \partial \phi_{G'-f} (v) = -\partial \phi_{G'-f}(v')$, say this value is $a$. But then we may extend $\phi$ to a nowhere-zero integer flow in $G'$ by setting $|\phi(f)| = a$.
\end{proof}

\begin{theorem}[folklore]\label{red to cubic}
    Let $k \ge 2$ be an integer. If $G$ is a flow-admissible signed graph without a nowhere-zero $k$-flow such that $\sum_{v \in V(G)}|d(v) - 2.5|$ is minimum, then $G$ is cubic.
\end{theorem}

\begin{proof} Because $G$ is flow-admissible, $G$ has no vertex of degree one. By minimality and Observation \ref{deg2}, $G$ is connected and has no vertex of degree two. Suppose, for contradiction, that $G$ has a vertex $v \in V(G)$ of degree at least 4, and let $e,e' \in \delta(v)$. Let $G'$ be the signed graph obtained from $G$ by uncontracting at $v$ with $e,e'$, and let $f$ be the uncontracted edge. By Lemma \ref{uncont}, one of $G'$ or $G' -f$ is flow-admissible, call it $G^*$. By minimality, $G^*$ has a nowhere-zero $k$-flow $\phi_k$. But $\phi_k$ restricted to $E(G)$ is a nowhere-zero $k$-flow in $G$, a contradiction. Thus $G$ has no vertex of degree at least 4, and so $G$ is cubic.
\end{proof} 

\section{Proof of main result}\label{S3}

In this section, we begin with a method to convert a $\Z_6$-flow to a 6-flow under specific circumstances in Lemma \ref{Z6 to 6}. Then we show that those circumstances can be achieved when the graph is cubic in Lemma \ref{sZ6}. Combining these lemmas, we prove Theorem \ref{mainCyc5ecBou}. 

If $e$ is a positive edge in a signed graph $G$, then we \emph{contract} $e$ by identifying its ends and deleting $e$ from $E(G)$. We denote the resulting graph by $G/ e$. A vertex $v \in V(G)$ is a \emph{source} if $\delta^-(v) = \emptyset$, and is a \emph{near-source} if $|\delta^-(v)| = 1$.

\begin{lemma}\label{Z6 to 6}
    Let $G$ be a cyclically 5-edge-connected signed graph with minimum degree 3. Suppose $G$ has an orientation $\tau$, and a nowhere-zero function $\phi: E(G) \to \Z$ such that (under orientation $\tau$)
    \begin{enumerate}
        \item $\phi(e) \in \{1,2,3,4,5\}$ for all $e \in E(G)$,
        \item $\partial \phi(v) \in \{0,6\}$ for all $v \in V(G)$,
        \item if $\partial \phi (v) = 6$ then $v$ is a source, if $\partial \phi(v) = 0$ then $v$ is a near-source, and
        \item there are an even number of source vertices.
    \end{enumerate}
    Then $G$ has an orientation $\tau^*$ where every vertex is a near-source, and a nowhere-zero 6-flow $\psi^*$ so that when $G$ has orientation $\tau^*$, then $\psi^*: E(G) \to \{1,2,3,4,5\}$.% and $\psi = \phi \mod 6$.
\end{lemma}

\begin{proof}
    Suppose, for contradiction, that the lemma is false, and let $G, \phi$, and $\tau$ be a counterexample such that the number of near-source vertices is minimum. 

    \medskip
    
    First, suppose there exists $e \in E(G)$ that is negative and oriented toward its ends $u,v$. By the assumptions of the lemma, this means $u,v$ are near-sources, $u \neq v$, and $\partial\phi(u) = 0 = \partial\phi(v)$. But then define orientation $\tau'$ from $\tau$ by reversing $e$, and define function $\phi'$ from $\phi$ by setting $\phi'(e) = 6-\phi(e)$. Then, $\partial\phi'(u) = 6 = \partial\phi'(v)$, and $u,v$ are sources in $\tau'$. Thus $G,\phi',\tau'$ satisfies the lemma by minimality. But this is a contradiction because $G, \phi', \tau'$ satisfies the lemma exactly when $G,\phi, \tau$ does, and $G, \phi, \tau$ was chosen to be a counterexample. This means every negative edge in $E(G)$ is oriented away from its ends.

    \medskip
    
    Next, suppose there exists $e \in E(G)$ that is positive, and let $e$ be directed from $u$ to $v$. This means that $v$ is a near-source and $\partial\phi(v) = 0$. In particular, by minimum degree 3, this means $e$ is not a loop. Let $G'= G/e$, and let $w$ be the vertex formed by identifying $u$ and $v$. Let $\tau'$ and $\phi'$ be $\tau$ and $\phi$, respectively, restricted to $G'$. Then $\partial\phi'(w) = \partial\phi(u) + \partial \phi(v) = \partial \phi(u)$, and $w$ has the same type as $u$ (source or near-source). It follows that $G', \phi'$, and $\tau'$ satisfy the assumptions of the lemma, but with one less near-source.
    
    By minimality applied to $G', \phi',\tau'$, we obtain an orientation $\tau^*$ of $G-e$ so that every vertex is a near-source but one, either $u$ or $v$; and an integer-valued function $\psi^*: E(G-e) \to \{1,2,3,4,5\}$ so that $\partial \psi^*(x) = 0$ when $x \not\in \{ u,v\}$, and $\partial \psi^*(v) = -\partial \psi^*(u)$, which we define to be $s$. Without loss, say $v$ is the source and $u$ is a near-source. Since $\psi^*$ takes only positive values, it follows that $s > 0$. 
    Let $f$ be the unique edge oriented toward $u$ in $\tau^*$. Because $\partial \psi^*(w) = 0$, it follows that $s \le \psi^*(f)$. 
    This means that $s \in \{1,2,3,4,5\}$. Thus extending $\tau^*$ by orienting $e$ from $u$ to $v$, and extending $\psi^*$ by setting $\psi^*(e) = s$ is a solution to the lemma for $G, \psi, \tau$, which contradicts that they were chosen to be a counterexample. Therefore every edge in $E(G)$ is negative.

    \medskip

    Finally, by our arguments so far, it is the case that every edge in $E(G)$ is negative and oriented away from its ends. Thus every vertex $v \in V(G)$ is a source, and by the assumptions of the lemma, $\partial\phi(v) = 6$. Let $G^*$ be the (unsigned, undirected) multigraph with vertex-set $V(G)$, and edge-set obtained by   including every edge $e \in E(G)$ in $E(G^*)$ exactly $\phi(e)$ times. Then $G^*$ is 6-regular and 5-edge-connected% (by vertex-degrees, if any edge-cut had less than 6 edges it would necessarily separate two cycles)
    . Since $G^*$ has an even number of (source) vertices, it has a perfect matching by Tutte's theorem \cite{tutteMatching}.\footnote{See also Exercise 16.4.9 in \cite{BondyJ.A.JohnAdrian2008Gt/J}.} This means $G$ also has a perfect matching $M$. Now, reversing every edge in $M$, and defining a nowhere-zero 6-flow $\psi^*: E(G) \to \Z$ by setting $\psi^*(e) = 6-\phi(e)$ if $e \in M$, and $\psi^*(e) = \phi(e)$ otherwise satisfies the lemma. This contradicts that $G, \phi$, and $\tau$ were chosen to be a counterexample, and completes the proof.
\end{proof}

Our next lemma relies on the following important result of DeVos et al., which was crucial in their proof of the 11-flow theorem. The essence of Lemma \ref{sZ6} is to show that the requirement on the edges given in Theorem \ref{11f} has a forcing on the vertex-boundaries. This forcing is what gives us an even number of source vertices.

For a function $f$ acting on the set $E$, we denote $\{e \in E: f(e) \neq 0\}$ by $\supp(f)$.

\begin{theorem}[\cite{DLLZZ}]\label{11f}
    Every flow-admissible signed graph has a nowhere-zero $\Z_2 \times \Z_3$-flow $\phi_2 \times \phi_3$ so that $\supp(\phi_2)$ contains an even number of negative edges.
\end{theorem}

%We remind the reader that the support of a $\Z_2$-flow is a subgraph where every vertex has even degree, and so Theorem \ref{11f} does not depend on the equivalent signature of the signed graph. We now prove the existence of a special nowhere-zero $\Z_6$-flow in a cubic signed graph $G$. To be more precise, we are showing that the requirement on the edges given in Theorem \ref{11f} has a forcing on the vertex-boundaries.

\begin{lemma}\label{sZ6}
    Every flow-admissible, cubic signed graph $G$ has an orientation $\tau$ and function $\phi: E(G) \to \Z$ such that (under $\tau$)
    \begin{enumerate}
        \item $\phi(e) \in \{1,2,3\}$ for all $e \in E(G)$,
        \item $\partial \phi(v) \in \{0,6\}$ for all $v \in V(G)$, 
        \item if $\partial \phi (v) = 6$ then $v$ is a source, if $\partial \phi(v) = 0$ then $v$ is a near-source, and
        \item there are an even number of source vertices.
    \end{enumerate}
\end{lemma}

\begin{proof}
    Let $\Sigma$ be the signature of $G$, and let $\phi_2 \times \phi_3 : E(G) \to \Z_2 \times \Z_3$ be a nowhere-zero flow so that $|\supp(\phi_2) \cap \Sigma|$ is even, which exists by Theorem \ref{11f}. Because $\Z_2 \times \Z_3$ and $\Z_6$ are isomorphic (we give the isomorphism below), we obtain a nowhere-zero $\Z_6$-flow $\phi_6$ so that $|\{e \in \Sigma : \phi_6(e) \in \{1,3,5\}|$ is even.  
    \begin{center}
    \begin{tabular}{ |c||c|c|c|c|c|c|} 
        \hline
        $\Z_2 \times \Z_3$ &(0,0)&(1,1)&(0,2)& (1,0) & (0,1) & (1,2)  \\ 
        $\Z_6$ & 0 & 1 &2&3&4&5\\ 
        \hline
    \end{tabular}
    \end{center}
    By reversing any edge $e$ with $\phi_6(e)\in\{4,5\}$ and replacing $\phi_6(e)$ with its inverse, we obtain an orientation $\tau$ of $G$ so that $\range(\phi_6) \subseteq \{1,2,3\}$. But then, this means there is an integer-valued function $\phi: E(G) \to \Z$ so that $\range(\phi) \subseteq \{1,2,3\}$, $|\{e \in \Sigma : \phi(e) \in \{1,3\}|$ is even, and $\partial \phi(v) = 0 \mod 6$ for all $v \in V(G)$. In fact, because $G$ is cubic and the range of $\phi$ is so limited, every vertex is one of the following six types with respect to $\tau$ and $\phi$.

    \begin{figure}[h]
        \centering
        \includegraphics[width=0.45\linewidth]{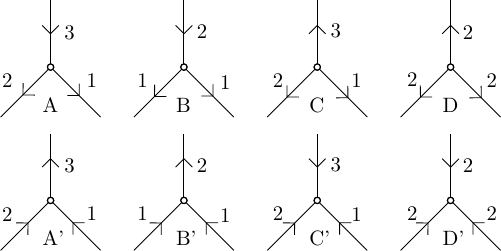}
    \end{figure}
    
    By switching at vertices of type A', B', C' and D', every vertex $v$ is either a source with $\partial \phi(v) = 6$, or a near-source with $\partial \phi(v) = 0$.

    \medskip
    It remains to show that the number $w$ of source vertices is even. By the paragraph above, it follows that $\sum_{v \in V(G)}\partial \phi(v) = 6w$.

    Now, consider the above sum from the point of view of the edges. Every positive edge $e$ contributes zero to the sum ($e$ is directed away from some vertex $v$ and towards some vertex $v'$, and so contributes $\phi(e)$ at $v$, and $-\phi(e)$ at $v'$). Every negative edge contributes $\pm2 \phi(e)$ to this sum (if $e$ is directed away from both its ends, it contributes $2\phi(e)$, and if $e$ is directed towards both its ends it contributes $-2\phi(e)$).  Define, for $i\in\{1,2,3\}$,
\begin{align*}
x_i
&=
\bigl|\{e\in\Sigma : \phi(e)=i \text{ and } e \text{ is directed away from its ends}\}\bigr|
\\
&\quad-
\bigl|\{e\in\Sigma : \phi(e)=i \text{ and } e \text{ is directed toward its ends}\}\bigr|.
\end{align*}
     Then, we obtain $\sum_{v \in V(G)}\partial \phi(v) = 2x_1+4x_2+6x_3$, and it follows that 
    \begin{equation}\label{wxyz}
        6w = 2(x_1+x_3) + 4x_2 + 4x_3. 
    \end{equation}
    But since $|\{e \in \Sigma : \phi(e) \in \{1,3\}|$ is even, it follows that $x_1 + x_3$ is even. This means the right hand side of Equation \ref{wxyz} is divisible by 4, and so the left side must also be divisible by 4. This means that $w$ is even, and completes the proof.
\end{proof}

\begin{proof}[Proof of Theorem \ref{mainCyc5ecBou}]
    Let $G$ be a flow-admissible, cyclically 5-edge-connected cubic signed graph. By Lemma \ref{sZ6}, $G$ has an orientation $\tau$ and a function $\phi$ satisfying the requirements of Lemma \ref{Z6 to 6}. By Lemma \ref{Z6 to 6}, $G$ has a nowhere-zero 6-flow.
\end{proof}

\bibliographystyle{abbrv}
\bibliography{bib}

@article {Bouchet,
    AUTHOR = {Bouchet, A.},
     TITLE = {Nowhere-zero integral flows on a bidirected graph},
   JOURNAL = {J. Combin. Theory Ser. B},
  FJOURNAL = {Journal of Combinatorial Theory. Series B},
    VOLUME = {34},
      YEAR = {1983},
    NUMBER = {3},
     PAGES = {279--292},
      ISSN = {0095-8956},
   MRCLASS = {05C15 (05C10 05C20)},
       DOI = {10.1016/0095-8956(83)90041-2},
       URL = {https://doi.org/10.1016/0095-8956(83)90041-2},
}

@article{Khelladi,
    author = {Khelladi, A.},
    title = {Nowhere-zero integral chains and flows in bidirected graphs},
    journal = {J. Combin. Theory Ser. B},
    year = {1987},
    volume = {43},
    pages = {95--115},
number = 1
}

@article {DLLZZ,
    AUTHOR = {DeVos, Matt and Li, Jiaao and Lu, You and Luo, Rong and Zhang,
              Cun-Quan and Zhang, Zhang},
     TITLE = {Flows on flow-admissible signed graphs},
   JOURNAL = {J. Combin. Theory Ser. B},
  FJOURNAL = {Journal of Combinatorial Theory. Series B},
    VOLUME = {149},
      YEAR = {2021},
     PAGES = {198--221},
      ISSN = {0095-8956},
   MRCLASS = {05C21 (05C22)},
       DOI = {10.1016/j.jctb.2020.04.008},
       URL = {https://doi.org/10.1016/j.jctb.2020.04.008},
}

@article{XU2005335,
title = {On flows in bidirected graphs},
journal = {Discrete Mathematics},
volume = {299},
number = {1},
pages = {335-343},
year = {2005},
issn = {0012-365X},
doi = {https://doi.org/10.1016/j.disc.2004.06.023},
url = {https://www.sciencedirect.com/science/article/pii/S0012365X05002852},
author = {Rui Xu and Cun-Quan Zhang},
}

@article{RASPAUD2011464,
title = {Circular flow on signed graphs},
journal = {J. Combin. Theory Ser. B},
volume = {101},
number = {6},
pages = {464-479},
year = {2011},
issn = {0095-8956},
doi = {https://doi.org/10.1016/j.jctb.2011.02.007},
url = {https://www.sciencedirect.com/science/article/pii/S0095895611000244},
author = {Andre Raspaud and Xuding Zhu},
}

@article{3flow8ec,
author = {Wu, Yezhou and Ye, Dong and Zang, Wenan and Zhang, Cun-Quan},
title = {Nowhere-Zero 3-Flows in Signed Graphs},
journal = {SIAM Journal on Discrete Mathematics},
volume = {28},
number = {3},
pages = {1628-1637},
year = {2014},
doi = {10.1137/130941687},
}

@article{Tutte_1954, title={A Contribution to the Theory of Chromatic Polynomials}, volume={6}, DOI={10.4153/CJM-1954-010-9}, journal={Canadian Journal of Mathematics}, author={Tutte, W. T.}, year={1954}, pages={80–91}}

@article{tutteMatching,
    author = {Tutte, W. T.},
    title = {The Factorization of Linear Graphs},
    journal = {Journal of the London Mathematical Society},
    volume = {s1-22},
    number = {2},
    pages = {107-111},
    year = {1947},
    month = {04},
    issn = {0024-6107},
    doi = {10.1112/jlms/s1-22.2.107},
    url = {https://doi.org/10.1112/jlms/s1-22.2.107},
    eprint = {https://academic.oup.com/jlms/article-pdf/s1-22/2/107/2624763/s1-22-2-107.pdf},
}

@misc{devos2023nowherezero8flowscyclically5edgeconnected,
      title={Nowhere-zero 8-flows in cyclically 5-edge-connected, flow-admissible signed graphs}, 
      author={Matt DeVos and Kathryn Nurse and Robert Sámal},
      year={2023, arXiv 2309.00704},
      eprint={2309.00704},
      archivePrefix={arXiv},
      primaryClass={math.CO},
      url={https://arxiv.org/abs/2309.00704}, 
}

@phdthesis{myThesis,
    author = {Kathryn Nurse},
    title = {On nowhere-zero flows in signed graphs},
    school = {Simon Fraser University},
    year = {2024},
}

@book{BondyJ.A.JohnAdrian2008Gt/J,
series = {Graduate texts in mathematics ; 244},
publisher = {Springer},
isbn = {9781846289699},
year = {2008},
title = {Graph theory},
address = {New York},
author = {Bondy, J. A. and Murty, U. S. R.},
lccn = {2007940370},
}

\end{document}